\documentclass{pspum-l}
\usepackage{wasysym}
\usepackage{amsmath}
\usepackage{amssymb}
\usepackage{amsthm}
\usepackage[cmtip,all]{xy}
\usepackage{MnSymbol}
\usepackage{tikz}
\usetikzlibrary{matrix,arrows,backgrounds,shapes.misc,shapes.geometric,patterns,calc,positioning}
\usetikzlibrary{calc,shapes}
\usepackage{wrapfig}
\usepackage{epsfig}  
\usepackage{color}	 %

\newtheorem{thm}{Theorem}[section]

\newtheorem{problem}[thm]{Problem}
    
\newtheorem*{thm*}{Theorem}

\theoremstyle{definition}
\newtheorem{definition}[thm]{Definition}
\newtheorem{example}[thm]{Example}
\theoremstyle{remark}
\newtheorem{remark}[thm]{Remark}
\newtheorem{convention}[thm]{Convention}
\numberwithin{equation}{section}

\newcommand{\cald}{\mathcal{D}}

\newcommand{\calg}{\mathcal{G}}

\newcommand{\za}{\alpha}
\newcommand{\zb}{\beta}

\newcommand{\zg}{\gamma}

\newcommand{\zo}{\omega}

\DeclareMathOperator{\Match}{Match}

\begin{document}

\title{Perfect matching problems in cluster algebras and number theory} 
\author{Ralf Schiffler}
\address{Department of Mathematics, University of Connecticut, Storrs, CT 06269-1009, USA}
\email{schiffler@math.uconn.edu}
\thanks{The author was supported by the NSF grant  DMS-2054561.}

\setcounter{tocdepth}{1}
\begin{abstract}

This paper is a slightly extended version of the talk I gave at the Open Problems in Algebraic Combinatorics conference at the University of Minnesota in May 2022.
 We introduce two strict order relations on lattice paths  and formulate several open problems. The topic is related to Markov numbers, the Lagrange spectrum, snake graphs and the cluster algebra of the once punctured torus.

Our lattice paths are required to proceed by North and East steps and never go over the diagonal. 
To define the order relations, we first construct a snake graph $\calg(\zo)$ and a band graph $\overline{\calg(\zo)}$ for every such lattice path $\zo$. The  first order relation $<_M$ is given by the number of perfect matchings of the snake graphs. The second order relation $<_L$ is given by the Lagrange number of a quadratic irrational associated to the band graph.

\end{abstract}

\maketitle
\tableofcontents
\section{Introduction}
We introduce two strict order relations on lattice paths $\zo$ that proceed by North and East steps and never go over the diagonal. 
To define the order relations, we first construct a snake graph $\calg(\zo)$ and a band graph $\overline{\calg(\zo)}$ for every lattice path $\zo$. 

Snake graphs and band graphs arise in the theory of cluster algebras from surfaces, where they are used to compute the Laurent expansion of the cluster variables and to construct canonical bases of the cluster algebra. A survey of these results is given in \cite{S}.

The lattice path snake graphs in this paper form a special class of snake graphs. Since the lattice can be thought of the universal cover of the torus with one puncture, (a slight deformation of) the lattice path corresponds to an element of the cluster algebra of the once-punctured torus. In particular, the Christoffel paths correspond to the cluster variables. 

The cluster variables are related to the Markov numbers. In our setting this relation can be expressed by saying that the number of perfect matchings of the snake graph of a Christoffel path is a Markov number and each Markov number arises this way.

A classical theorem of Markov gives a construction of the Lagrange spectrum below 3 in terms of the Markov numbers. Our order relation $<_M$ is related to   Markov numbers and the relation $<_L$ is related to the Lagrange spectrum.

We review  snake graphs and band graphs in section~\ref{sect snakegraphs} and Markov numbers and the Lagrange spectrum in section~\ref{sect Lagrange}. We define the two order relations in  section~\ref{sect main}. In section~\ref{sect known}, we list known results about the order relations and section~\ref{sect open} contains the open problems.

I would like to thank the organizers of the OPAC conference for their invitation and encouragement to write up these notes.

\section{Snake graphs and band graphs}
\label{sect snakegraphs}
Snake graphs and band graphs were used in \cite{MS,MSW,MSW2} to construct canonical bases for cluster algebras from surfaces. The Laurent expansions of the basis elements are parametrized by the perfect matchings of these graphs.
In \cite{MOZ}, some of these results were extended to the super cluster algebra setting by replacing the perfect matchings of snake graphs by double dimer covers. 

\subsection{Snake graphs}
We now recall the definition of snake graphs following \cite{CS4}. 
A \emph{tile} is a planar graph with 4 vertices and 4 edges that has the shape of a square. All tiles will have the same side length. 
A snake graph $\calg$ is a connected planar graph consisting of a finite sequence of tiles $G_1,\ldots, G_d$ that is constructed recursively by the rule that the $i+1$-st tile $G_{i+1}$ shares exactly one edge $e_i$ with the previous tile $G_{i}$, and this edge is either the north edge of $G_i$ and the south edge of $G_{i+1}$ or the east edge of $G_i$ and the west edge of $G_{i+1}$. 
An example is given in Figure \ref{signfigure}.
\begin{figure}
\begin{center}
  {\tiny \scalebox{0.9}{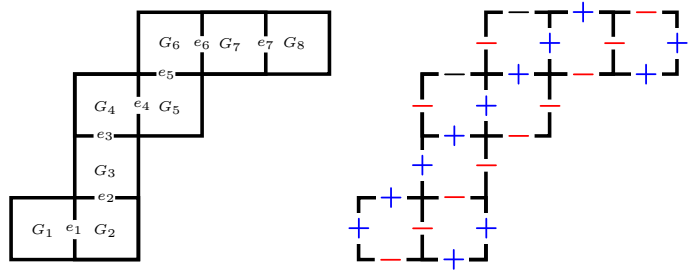}}
 \caption{A snake graph with 8 tiles and 7 interior edges (left);
 a sign function on the same snake graph (right)} 
 \label{signfigure}
\end{center}
\end{figure}

The $d-1$ edges $e_1,e_2, \dots, e_{d-1}$ which are contained in two tiles are called {\em interior edges} of $\calg$ and the other edges are called {\em boundary edges.}  
We denote by  
 $\calg^{N\!E}$ the 2 element set containing the north edge and the east edge of the last tile of $\calg$.

A {\em sign function} $f$ on a snake graph $\calg$ is a map $f$ from the set of edges of $\calg$ to the set $\{ +,- \}$ such that on every tile in $\calg$ the north and the west edge have the same sign, the south and the east edge have the same sign and the sign on the north edge is opposite to the sign on the south edge. See Figure \ref{signfigure} for an example.
%
%

Note that on every snake graph  there are exactly two sign functions. A snake graph is  determined up to symmetry by its sequence of tiles together with a sign function on its interior edges.

\subsection{Band graphs} Every snake graph  $\calg$ gives rise to two \emph{band graphs} $\overline{\calg}$ defined by identifying either
\begin{itemize}
\item [-]  the south edge of the first tile in $\calg$ with the unique edge in $\calg^{N\!E}$ that has the same sign;
\item[-] or the west edge of the first tile in $\calg$ with the unique edge in $\calg^{N\!E}$ that has the same sign.
 \end{itemize}

\subsection{Perfect matchings} A \emph{perfect matching} of a graph $\calg$ is a subset $P$ of the 
edges of $\calg$ such that
each vertex of $\calg$ is incident to exactly one edge of $P$. 
We denote by $\Match\calg$ the set of  perfect matchings of $\calg$.
If $\overline{\calg}$ is a band graph, we define $\Match \overline{\calg}$ to be the set of all perfect matchings $P$ of the snake graph $\calg$ such that $P$ is a perfect matching of $\overline{\calg}$.

\subsection{Relation to continued fractions}
A relation between snake graphs and continued fractions was discovered in \cite{CS4}.

A \emph{finite continued fraction} is a function 
\[[a_1,a_2,\ldots,a_n]= a_1+\cfrac{1}{a_2+\cfrac{1}{\ddots +\cfrac{1}{a_n}}}\]
of $n$ positive integers $a_1,\ldots,a_n$. Notice that $[a_1,\ldots,a_{n-1},1]=[a_1,\ldots,a_{n-1}+1]$.
Similarly, an \emph{infinite continued fraction} is a function 
\[[a_1,a_2,\ldots]= a_1+\cfrac{1}{a_2+\cfrac{1}{a_3+\cfrac{1}{\ddots }}}\]
of infinitely many positive integers $a_1,a_2,\ldots$. 
We use the notation $[\overline{a_1,\ldots, a_n}]$ for the periodic continued fraction $[\overline{a_1,\ldots, a_n}]=[{a_1,\ldots, a_n},{a_1,\ldots, a_n},\ldots]$.
 
 Now let $\calg$ be a snake graph with $d$ tiles and sign function $f$. Let $e_1,\ldots, e_{d-1}$ be the sequence of interior edges, let $e_0$ be the south or west edge of the first tile and let $e_d$ be the north or east edge of the last tile. The sequence of signs $(f(e_i))_{i=0}^d$ is called the sign sequence of $\calg$. It determines the snake $\calg$ completely.

\begin{convention}
{In this paper, we shall always choose the edge $e_0$ to be the west edge of the first tile and the edge $e_d$ to be such that the last two signs in the sign sequence are equal. }
\end{convention}
  
 The sign sequence in the example in Figure \ref{signfigure}  is  $(+,-,-,+,+,+,+,-,-)$.
 \smallskip
 
 To the sign sequence of a snake graph, we also associate a sequence of positive integers $a_1,\ldots,a_n$ that count the lengths of maximal constant subsequences in the sign sequence. In the example, this sequence  $a_1,\ldots,a_n$ is equal to $1,2,4,2$. 
 The sequence $a_1,\ldots,a_n$ uniquely determines the snake graph and we therefore denote the snake graph by $\calg[a_1,\ldots,a_n]$. 
 
It is shown in \cite[Theorems 3.4 and 4.1]{CS4} that 
the construction above gives a bijection between   snake graphs and continued fractions (whose last coefficient is greater than one), such that the continued fraction 
is computed by the following formula
\begin{equation}
\label{eq cf}
 [a_1,\ldots,a_n] =\frac{|\Match(\calg[a_1,\ldots,a_n])|}{|\Match(\calg[a_2,\ldots,a_n])|}.
\end{equation}
This is a combinatorial interpretation of continued fractions in terms of cardinalities of sets of perfect matchings. 

The snake graph in Figure~\ref{signfigure} has 29 perfect matchings, and we have $[a_1,\ldots,a_n]=[1,2,4,2]=\frac{29}{20}$.

\begin{problem}\textup{(Solved)}
 Find an analogue of  (\ref{eq cf}) for band graphs. 
\end{problem} 
This problem has been solved by PJ Apruzzese as follows. 
Given a band graph, associate a sign sequence to it considering only the interior edges, and denote the corresponding integer sequence by $a_0,a_1,\ldots,a_n$ as usual. 
The corresponding periodic continued fraction $[\overline{a_0,a_1,\ldots,a_n}]=[a_0,a_1,\ldots,a_n,a_0,a_1,\ldots,a_n,a_0,a_1,\ldots]$ is equal to a quadratic irrational \[[\overline{a_0,a_1,\ldots,a_n}]= \frac{p-s+\sqrt{(p-s)^2+4rq}}{2q},\] where $p/q=[a_0,a_1,\ldots,a_n], $ and $r/s=[a_0,a_1,\ldots,a_{n-1}]$. Let $D={(p-s)^2+4rq}$ be the radical.
Apruzzese obtains the following two formulas. 
\begin{thm}\label{thm pj}
 \cite{Ap}
Let $\calg[\overline{a_0,a_1,\ldots,a_n}]$ be a band graph. Then
\[
\begin{array}{rcl}
\mid\Match \calg[\overline{a_0,a_1,\ldots,a_n}]\mid &=& p+s 
\\ 
&=& \sqrt{D+4}.
\end{array}
\]
\end{thm}
\begin{remark}
The sequence $a_0,a_1,\ldots,a_n$ depends on the choice of the snake graph $\calg$ that realizes the given band graph. Indeed, cutting the band graph an any interior edge will produce such a snake graph. However, the radical $D$ and the sum $p+s$ doesn't depend on this choice.
\end{remark}
\begin{example}
 Starting from the snake graph $\calg[1,2,4,2] $ in Figure~\ref{signfigure}, we can form the two band graphs $\calg\left[\overline{4,4}\right]$ and $\calg\left[\overline{1,2,4,1}\right]$. They are obtained from the snake graph by identifying the south (respectively west) edge of the first tile with the north (respectively east) edge of the last tile.
 
 For $\calg\left[\overline{4,4}\right]$, we have $\frac{p}{q}=[4,4] =\frac{17}{4}$, 
 $\frac{r}{s}=[4] =\frac{4}{1}$, thus $|\Match\calg\left[\overline{4,4}\right]|=p+s=18.$
 Using the radical formula instead, we have $D=(p-s)^2+4rq=16^2-4^3=320$ and thus $\sqrt{D+4}=18$, confirming the result.

 For $\calg\left[\overline{1,2,4,1}\right]$, we have $\frac{p}{q}=[1,2,4,1] =\frac{16}{11}$, 
 $\frac{r}{s}=[1,2,4] =\frac{13}{9}$, and therefore $|\Match\calg\left[\overline{1,2,4,1}\right]|=p+s=25.$
 On the other hand, $D=(p-s)^2+4rq=7^2-4\cdot13\cdot 11=621$ and thus $\sqrt{D+4}=25$, again confirming the result. 
\end{example}

\begin{problem}
 Find an analogue  of (\ref{eq cf})  for double dimer covers.
\end{problem}

%
\section{Markov numbers and  the Lagrange spectrum}\label{sect Lagrange}

In 1879,   Markov studied the equation
\begin{equation}
 \label{eq M} 
 x^2+y^2+z^2=3xyz,
\end{equation}
 now known as the \emph{Markov equation}.  A positive integer solution $(m_1,m_2,m_3)$ of (\ref{eq M}) is called a \emph{Markov triple} and the integers that appear in the Markov triples are called \emph{Markov numbers}. 
 For example $(1,1,1),(1,1,2),(1,2,5),(1,5,13),(2,5,29)$ are Markov triples and $1,2,5,13,29$ are Markov numbers.
 
In 1913, Frobenius conjectured that, for every Markov number $m$, there exists a unique Markov triple in which $m$ is the largest number, see \cite{F}. This uniqueness conjecture is still open today. It has inspired a considerable amount of research and the Markov numbers have important ramifications in number theory, hyperbolic geometry, combinatorics and algebraic geometry. For an overview, we refer to the recent textbooks \cite{A,R}.

  The Markov numbers can be represented in a binary tree called the \emph{Markov tree}. This Markov tree is combinatorially equivalent to the Farey or Stern-Brocot tree of rational numbers. Thus there is a correspondence between $\mathbb{Q}_{[0,1)}$ and the Markov numbers by considering corresponding positions in these trees. We henceforth will write $m_0=1$ and refer to all other Markov numbers as $m_{\frac{b}{a}}$, where $b<a$ are relatively prime positive integers.
For example,
$
m_{\frac{1}{1}}=2,
m_{\frac{1}{2}}=5,
m_{\frac{1}{3}}=13,
m_{\frac{2}{3}}=29$.

The Markov numbers are related to the Lagrange numbers in  approximation theory. Given a real number $\za$, its \emph{Lagrange
number} $L(\za)$ is defined as the supremum of all real numbers $L$ for which there exist infinitely many rational numbers $\frac{p}{q}$ such that 
$|\za-\frac{p}{q} | < \frac{1}{Lq^2}$. Thus the Lagrange number measures how well the real number $\za$ can be approximated by rational numbers.

The \emph{Lagrange spectrum} is defined as the set of all Lagrange numbers $L(\za)$, where $\za$ ranges over all irrational real numbers. Considering it as a subset of the real line, it is known that the Lagrange spectrum is discrete below 3, it is fractal between 3 and a number $F\approx 4.5278$ called the \emph{Freiman number}, and it is continuous above $F$, see Figure~\ref{figlagrange}. We recommend \cite{CF} as reference for the Lagrange spectrum.

 Markov proved the following theorem in 1879.
 
\begin{thm}
 \cite{M} The Lagrange spectrum below 3 is precisely the set of all 
 $\sqrt{9m^2-4} / m $, where $m$ ranges over all Markov numbers.
\end{thm}

Note that $m<m'$ implies $\sqrt{9m^2-4} / m<\sqrt{(9m')^2-4} / m'$.
The first four Lagrange numbers are therefore
$\sqrt{5},\sqrt{8},\frac{\sqrt{221}}{5}\approx 2.973,\frac{\sqrt{1517}}{13} \approx 2.996.$
\begin{figure}
\begin{center}
\scalebox{0.9}{
\begingroup%
  \makeatletter%
  \providecommand\color[2][]{%
    \errmessage{(Inkscape) Color is used for the text in Inkscape, but the package 'color.sty' is not loaded}%
    \renewcommand\color[2][]{}%
  }%
  \providecommand\transparent[1]{%
    \errmessage{(Inkscape) Transparency is used (non-zero) for the text in Inkscape, but the package 'transparent.sty' is not loaded}%
    \renewcommand\transparent[1]{}%
  }%
  \providecommand\rotatebox[2]{#2}%
  \newcommand*\fsize{\dimexpr\f@size pt\relax}%
  \newcommand*\lineheight[1]{\fontsize{\fsize}{#1\fsize}\selectfont}%
  \ifx\svgwidth\undefined%
    \setlength{\unitlength}{386.24319752bp}%
    \ifx\svgscale\undefined%
      \relax%
    \else%
      \setlength{\unitlength}{\unitlength * \real{\svgscale}}%
    \fi%
  \else%
    \setlength{\unitlength}{\svgwidth}%
  \fi%
  \global\let\svgwidth\undefined%
  \global\let\svgscale\undefined%
  \makeatother%
  \begin{picture}(1,0.11200641)%
    \lineheight{1}%
    \setlength\tabcolsep{0pt}%
    \put(0,0){\includegraphics[width=\unitlength,page=1]{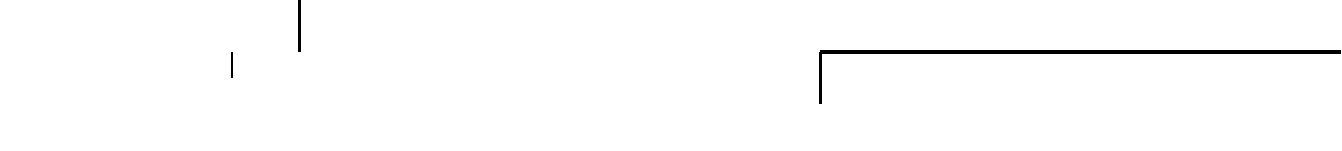}}%
    \put(0.21552015,0.00520841){\makebox(0,0)[lt]{\lineheight{1.25}\smash{\begin{tabular}[t]{l}3\end{tabular}}}}%
    \put(0.54562331,0.00520841){\makebox(0,0)[lt]{\lineheight{1.25}\smash{\begin{tabular}[t]{l}$F\approx 4.5278$\end{tabular}}}}%
    \put(0.39222227,0.08676321){\makebox(0,0)[lt]{\lineheight{1.25}\smash{\begin{tabular}[t]{l}fractal\end{tabular}}}}%
    \put(0.77086938,0.08676321){\makebox(0,0)[lt]{\lineheight{1.25}\smash{\begin{tabular}[t]{l}continuous\end{tabular}}}}%
    \put(0.07959552,0.08676321){\makebox(0,0)[lt]{\lineheight{1.25}\smash{\begin{tabular}[t]{l}discrete\end{tabular}}}}%
    \put(0.14367421,0.00909206){\makebox(0,0)[lt]{\lineheight{1.25}\smash{\begin{tabular}[t]{l}$\sqrt{8}$\end{tabular}}}}%
    \put(-0.00195948,0.0071502){\makebox(0,0)[lt]{\lineheight{1.25}\smash{\begin{tabular}[t]{l}$\sqrt{5}$\end{tabular}}}}%
    \put(0,0){\includegraphics[width=\unitlength,page=2]{figlagrangespectrum.pdf}}%
  \end{picture}%
\endgroup%
}
\caption{A schematic picture of the Lagrange spectrum}
\label{figlagrange}
\end{center}
\end{figure}

%
\section{Lattice path snake graphs}\label{sect main}

Let $a,b$ be relatively prime integers such that $0<b<a$. We denote by $\cald(a,b)$  the set of all lattice paths $\zo$ from $(0,0)$ to $(a,b)$ that proceed by North (Up) and East (Right) steps and never go above the diagonal. 
Let $\cald=\cup_{(a,b)} \cald(a,b)$.

\begin{definition}
 For $\zo\in\cald(a,b)$, we construct a snake graph $\calg(\zo)$ and a band graph $\overline{\calg(\zo)}$ as follows. An example is shown in Figure~\ref{figdefex}.

First place tiles of side length 1/2 along $\zo$. Then do one of the following.
\begin{itemize}
\item [-] To construct $\calg(\zo)$, remove the first and the last tile.
\item [-] To construct $\overline{\calg(\zo)}$, add one tile on top of the last tile and glue the left vertical edge of the first tile to the right vertical edge of the last tile. 
 \end{itemize}
\end{definition}

\begin{figure}\begin{center}
\tiny\scalebox{1.2}{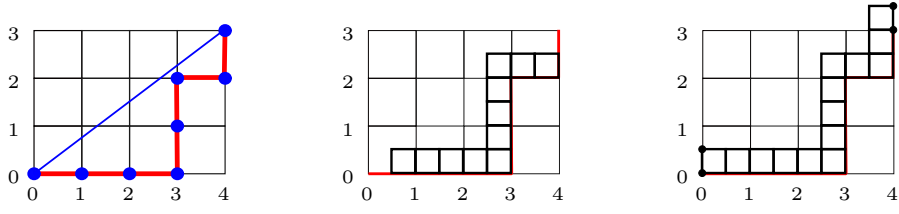}
\caption{The lattice path $\zo=\textup{RRRUURU}$ in $\cald(4,3)$ is shown in the picture on the left. The center picture shows its snake graph $\calg(\zo)$, and the picture on the right shows its band graph $\overline{\calg(\zo)}$, where the gluing edge is indicated by bullet points. }
\label{figdefex}
\end{center}
\end{figure}
In the example in Figure~\ref{figdefex}, we have 
$\zo=\textup{RRRUURU}$,
\[\calg(\zo)=\calg[1,1,1,1,2,1,1,2,2]
 \quad\textup{ and } \quad
\overline{\calg(\zo)}=\calg[\overline{2,1,1,1,1,2,1,1,2,2}] .\]
For an arbitrary lattice path $\zo\in\cald$, if  $\calg(\zo)=\calg[a_1,a_2,\ldots,a_n]$ then 
$\overline{\calg(\zo)}=\calg[\overline{2, a_1,a_2,\ldots,a_n}]$. 

Alternatively, we can construct the sequence of integers $a_1,\ldots ,a_n$ directly from the sequence of steps in $\zo$ by associating 
the subsequence $1,1$ to every subsequence RR  and every subsequence DD in $\zo$; and associating the subsequence 2 to every subsequence RU  and every subsequence UR in $\zo$.  

\begin{remark}
 These graphs are significant in the cluster algebra of the once-punctured torus. Via the covering map from the lattice to the torus, a slight deformation of the  path $\zo$ avoiding all intermediate lattice points corresponds to an arc  $\zg$ on the torus that starts and ends at the puncture.  This arc has selfcrossings unless it is the diagonal line from $(0,0)$ to $(a,b)$ and $\zo$ is the path closest to the diagonal.  The arcs without selfcrossings correspond to  cluster variables and the selfcrossing arcs to so-called generalized cluster variables. In both cases their Laurent expansion is computed by the snake graph $\calg(\zo)$. 
 The band graph $\overline{\calg(\zo)}$ computes the Laurent expansion a the closed curve $\overline{\zg}$ obtained from $\zg$ by moving its endpoints infinitesimally away from the puncture.
\end{remark}
\subsection{Two strict partial orders on lattice paths}
The following definition associates  two numbers to every lattice path $\zo\in\cald$: an integer $M(\zo)$ and a real number $L(\zo)$.
\begin{definition}
 Let $\zo\in\cald$ be a lattice path and $\calg(\zo)=\calg[a_1,\ldots,a_n]$ its snake graph.
 Define
 \[
\begin{array}{rcl}
  M(\zo) & =  & |\Match \calg(\zo)| = \textup{Numerator of }[a_1,\ldots, a_n] ;  \\ \\
 L(\zo) & =  & \textup{Lagrange number of } [\overline{2, a_1,\ldots,a_n}]   .
\end{array}
\]
\end{definition}
\begin{remark} \label{rem lagrange}
  The number $L(\zo)$ can be computed by the formula $L(\zo)=\max(\za-\za')$, where $\za$ runs over all cyclic shifts of  $[\overline{2, a_1,\ldots,a_n}] $ and $\za'$ is the conjugate of $\za$,
see \cite[Proposition 1.29]{A}. Recall that the conjugate of a quadratic irrational $\frac{a+\sqrt{D}}{b}$ is defined as  $\frac{a-\sqrt{D}}{b}$.
\end{remark}
This definition allows us to introduce two strict partial orders on lattice paths. Recall that a strict partial order is a relation $<$ that is irreflexive ($a\not < a$), asymmetric  and transitive.

\begin{definition}
 \label{def partial order}
 Let $\zo,\zo'\in \cald$ be two lattice paths. 
 
 (a) The \emph{matching order} relation $<_M$ is defined by 
 $\zo <_M \zo'$ if $M(\zo)<M(\zo')$.
 
 (b) The \emph{Lagrange order} relation $<_L$ is defined by 
 $\zo <_L \zo'$ if $L(\zo)<L(\zo')$.
 \end{definition}

\begin{remark}
 The order relations are strict partial orders $<$ instead of non-strict partial orders $\le$, because there can be two different lattice paths $\zo\ne \zo'$ such that $M(\zo)=M(\zo')$. Thus the non-strict partial order would not satisfy the antisymmetry condition.
\end{remark}
\begin{remark}
 By definition, any two lattice paths $\zo,\zo' $ are comparable under $<_M$, or $<_L$, unless $M(\zo)=M(\zo')$, or $L(\zo)=L(\zo')$, respectively.
\end{remark}
\begin{problem}
 Study the posets $\cald(a,b)$ and $ \cald$ with respect to $<_L$ and $<_M$.
\end{problem}
\section{Known results}\label{sect known}
The path in $\cald(a,b)$ that is closest to the diagonal is called the \emph{Christoffel path} of slope $b/a$.

\begin{thm} \label{thm M}

Let $\zo_0\in\cald(a,b)$ denote the Christoffel path.
Then $M(\zo_0)$ is the Markov number $m_{\frac{b}{a}}$.
Moreover $\zo_0$ is the unique minimal element in the poset $(\cald(a,b),<_M)$.
%

\end{thm}
\begin{proof}
This  follows from the work of several people. A correspondence between Markov numbers and the cluster algebra of a torus with one puncture was obtained in \cite{BBH,P}. Therefore the equation $M(w_0)=m_{\frac{b}{a}}$ follows from the  expansion formula for cluster variables in \cite{MSW}. The description of the snake graph $M(\zo_0)$ we use here first appeared in  \cite{CS5}.
The fact that $w_0$ is the unique minimal element follows from \cite[Theorem 3.5]{LLRS}.
%
%
%
%
%
%
\end{proof}

\begin{thm}
 \label{thm L}
Let $\zo\in\cald$. Then 
$L(\zo)<3$ if and only if $\zo$ is a Christoffel path.
In that case, $L(\zo)=\sqrt{9 m^2-4}/m$ as in Markov's theorem, where $m=M(\zo)$ is the Markov number of the Christoffel path. In particular, the Christoffel path is the unique minimal element in the poset $(\cald(a,b),<_L)$.
\end{thm}

\begin{proof}
 
 This follows from  a more general result by Markov which is restated (and proved) in terms of Christoffel words in \cite[Theorem 8.4]{BLRS}. The theorem states that for any bi-infinite sequence $A$ of positive integers, a certain quantity $M(A)<3$ if and only if $A$ is the bi-infinite repetition of a Christoffel word $\zo'$. Moreover, in that case,  $M(A)=\sqrt{rc^2-4}/c$, where $c$ is a Markov number associated to $\zo'$. 

Our setting is a special situation of this result, because we take the sequence $A$ to be the bi-infinite sequence $\ldots,2,a_1,\ldots,a_n,2,a_1,\ldots,a_n,\ldots$, where the band graph $\overline{\calg(\zo)}$ is $\calg[\overline{2,a_1,\ldots,a_n}]$. Then the quantity $M(A)$ becomes $\max(\za-\za')$, where $\za$ runs over all cyclic shifts of $[\overline{2,a_1,\ldots,a_n}]$. Thus $M(A)=L(\zo)$ according to Remark~\ref{rem lagrange}. 

We should point out that there is a slight difference in notation here. The $\zo'$ in \cite{BLRS} is not the same as the $\zo$ in our setting. Instead, $\zo$ is the image of $\zo'$ under the morphism $\mathbf{G}$ in \cite[Lemma 2.2]{BLRS}, which maps the Christoffel path in $\cald(a,b)$ to the Christoffel path in $\cald(a+b,b)$.

The statement that the Christoffel path is the unique minimal element in $(\cald(a,b),<_L)$ follows because the Christoffel path is the only path $\zo \in\cald(a,b)$ for which $L(\zo)<3$.
\end{proof}

\section{Open problems}\label{sect open} In this section, we give a list of open problems and some of them are followed by comments. 
Throughout we denote by $\zo_0(a,b)$ the Christoffel path in $\cald(a,b)$.

\medskip

\begin{problem}\label{prob 1}
Show that the unique maximal element in $\cald(a,b)$ with respect to both partial orders is the $\righthalfcup$-shape path $\zo=\textup{RR}\cdots\textup{RUU}\cdots \textup{U}
=\textup{R}^a\textup{U}^b$.
\end{problem}
This path $\zo$ first runs horizontally along the lower boundary of the $(a,b)$-rectangle and then runs vertically along the right boundary of the rectangle. This is the  path furthest from the diagonal.

\medskip

\begin{problem}\label{prob 2}
 
 Consider two lattice points $(a,b), (a',b')$ and their respective Christoffel paths  $\zo_0(a,b)$ and $\zo_0(a',b')$. Is it true that
\begin{enumerate}
\item [(a)]
$M(\zo_0(a,b))=M(\zo_0(a',b'))$ if and only if $ (a,b)=(a',b') $?
\item[(b)] $L(\zo_0(a,b))=L(\zo_0(a',b')) $ if and only if $ (a,b)=(a',b') ?$

\end{enumerate}
\end{problem}
Part (a) is equivalent to the Uniqueness Conjecture, because of Theorem~\ref{thm M}. 

For part (b), let us point out that two irrational numbers $\za,\zb$ are said to be equivalent if their continued fractions eventually coincide.  It is known that equivalent irrational numbers have the same Lagrange number. This is the case for the cyclic shifts of $[\overline{2,a_1,\ldots,a_n}]$ that occur in the definition of the $L(w_0(a,b))$ when the point $(a,b)$ is fixed.  
Note however, that in this problem we are dealing with two different points $(a,b)$ and $(a',b')$. 

It is a conjecture that $\za$ and $\zb$ are equivalent irrational numbers if and only if the Lagrange numbers $L(\za)$ and $ L(\zb)$ are equal. This conjecture is known to be equivalent to the Uniqueness Conjecture, see for example \cite[page 39]{A}.   

\medskip

\begin{problem}\label{prob 3}
  If $\zo,\zo'\in\cald(a,b)$ is it true that
  
$$L(\zo)=L(\zo') \quad \textup{ implies }\quad \overline\calg(\zo)\cong\overline\calg(\zo')?$$
\end{problem}
Note that the previous Problem \ref{prob 2} is not a special case of this problem, since there we  compared paths with different endpoints, whereas here the paths all have the same endpoint. 

 In Figure~\ref{sameLvalue}, we give an example of two lattice paths $\zo_1=$RRRRUURRRUURU and $\zo_2=$ RRRRURUURRRUU in $\cald(8,6)$ such that $L(\zo_1)=L(\zo_2)$. In this case, the corresponding band graphs 

\[\begin{array}{rcl}
\overline{\calg(\zo_1)}&=&\calg[\overline{2,1, 1, 1, 1, 1, 1, 2, 1, 1, 2, 1, 1, 1, 1, 2, 1, 1, 2, 2}] \\ 
\overline{\calg(\zo_2)}&=&\calg[\overline{2,1, 1, 1, 1, 1, 1, 2, 2, 2, 1, 1, 2, 1, 1, 1, 1, 2, 1, 1}] \\
\end{array}\]
are isomorphic, which can be seen from the fact that one obtains the periodic integer sequence of the second graph by the one of the first graph by reversing the order and shifting. For completeness, we also note that $M(\zo_1)=49396$ is not equal to $M(\zo_2)=46900$.
\begin{figure}\begin{center}
\tiny\scalebox{1.2}{\input{sameLvalue.pdf_tex}}
\caption{Two lattice paths $\zo_1,\zo_2$ in $\cald(8,5)$ such that $L(\zo_1)=L(\zo_2)$. The snake graphs $\overline{\calg(\zo_1)} $ and $\overline{\calg(\zo_2) }$ are isomorphic.}
\label{sameLvalue}
\end{center}
\end{figure}

\medskip

On the other hand, let us also mention  that the statement $(M(\zo)=M(\zo'))\Rightarrow (\calg(\zo)\cong\calg(\zo'))$
is false! For example the lattice paths $\zo=$ RRURURRURRUUU  and $\zo'=$ RRRURURRUURUU shown in Figure~\ref{sameMvalue} have the same value under the function $M$, namely 
40199. The continued fractions are \[[1,1,2,2,2,2,1,1,2,2,1,1,2,1,1,1,1]=\frac{40199}{23549} \]and \[ [1,1,1,1,2,2,2,2,1,1,2,1,1,2,2,1,1]=\frac{40199}{24653}.\]
We further point out that, if we measure the distance from the diagonal by the number of lattice points between the lattice path and the diagonal, then  the two paths do not  have the same distance from the diagonal. It is 4 for $\zo$ and 6 for $\zo'$. Moreover, the two paths have different Lagrange numbers 
\[L(\zo)=\frac{\sqrt{16530502037}}{40189} \quad \textup{and}\quad
L(\zo')=\frac{\sqrt{16545934157}}{40547} .
\]
\begin{figure}\begin{center}
\tiny\scalebox{1.2}{\input{sameMvalue.pdf_tex}}
\caption{Two lattice paths $\zo,\zo'$ in $\cald(7,6)$ such that $M(\zo)=M(\zo')$. }
\label{sameMvalue}
\end{center}
\end{figure}

\medskip

\begin{problem}\label{prob 4} Let $\zo,\zo'\in\cald$.

 (a) Is it true that $\zo<_L\zo'$ implies 
 $\zo<_M\zo'$ or $M(\zo)=M(\zo')$ ?
 
 (b) Is it true that $\zo<_M\zo'$ implies 
 $\zo<_L\zo'$ or $L(\zo)=L(\zo')$?
 
\end{problem}

\medskip

 \begin{problem}\label{prob 5}
 Give a characterization of the cover relation in the poset $\cald(a,b)$ for any of the two orders in terms of the lattice paths. 
 
\end{problem}
As a first step towards this problem, one can ask if $\zo$ is smaller than $\zo'$  for both orders if there exist lattice paths $\za,\zb$ such that $\zo= \za $RU$ \zb$ and 
$\zo= \za$UR$\zb$.
\medskip

\begin{problem}\label{prob 6}
 Prove that 
 $\sup\{ L(\zo)\mid \zo\in \cald\} = 1+\sqrt{5}.$
\end{problem}
There is some evidence for the above conjecture from  computer calculations.

\medskip

 \begin{problem}\label{prob 7}
 By definition the set $\{L(\zo)\mid \zo\in \cald\}$ is a subset of the Lagrange spectrum, and it contains the Lagrange spectrum below 3 by Theorem~\ref{thm L}.
 What part of the Lagrange spectrum above 3 does it contain? 
  
\end{problem}
  
  Note that upper bound conjectured in Problem \ref{prob 6} is equal to $1+\sqrt5\approx 3.23607$ which lies within the fractal part of the Lagrange spectrum. Assuming the conjecture is true, can every Lagrange number below $1+\sqrt{5}$ be realized as the limit of a sequence of Lagrange numbers $L(\zo)$ of lattice paths $\zo \in\cald$ ?
  
As an example, we show that the number 3 can be realized in this way. Indeed, let $\za$ be given by the infinite continued fraction
\[
\begin{array}
 {rcl}\za
&=& [(1,1)^1,2,2,(1,1)^2,2,2,(1,1)^3,2,2,(1,1)^4,2,2,\ldots],
\end{array}\]
where the notation $(1,1)^a$ stands for the sequence $1,\ldots,1$ of $2a$ numbers 1. 
This continued fraction is associated to  the infinite lattice path
\[\zo=\textup{R}^2\textup{U\,R}^3\textup{U\,R}^4\textup{U\,R}^5\textup{U}\cdots.\]
The Lagrange number of $\za$ is equal to 
\[L(\za)=\limsup_{n\to \infty}\big([a_n,a_{n+1},\ldots]+[0,a_{n-1},a_{n-2},\ldots,a_1]\big).\]
To compute the supremum we may assume that $n$ is such that $a_n=2$. There are two cases.  Consider first the case where $a_n$ is the second number in a $2,2$ subsequence. 
\[ \big([2,(1,1)^b,2,2,(1,1)^{b+1},\ldots]+[0,2,(1,1)^{b-1},(1,1)^{b-2},2,2,\ldots,2,2,1,1]\big)\]
As $b$ goes to infinity, the first continued fraction goes to $[2,\overline{1}] = 1+[\overline{1}]=\frac{3+\sqrt{5}}{2},$ and the second continued fraction goes to $[0,2,\overline{1}]=\frac{2}{3+\sqrt{5}}$. The sum of the two terms is 
$\frac{3+\sqrt{5}}{2}+\frac{2}{3+\sqrt{5}} =\frac{18+6\sqrt{5}}{6+2\sqrt{5}} =3$.
Now consider   the case where $a_n$ is the first number in a $2,2$ subsequence.
\[ \big([2,2,(1,1)^b,2,2,(1,1)^{b+1},\ldots]+[0,(1,1)^{b-1},(1,1)^{b-2},2,2,\ldots,2,2,1,1]\big).\]
As $b$ goes to infinity, the first continued fraction goes to $[2,2,\overline{1}] = 2+\frac{2}{3+\sqrt{5}}=\frac{8+2\sqrt{5}}{3+\sqrt{5}},$ and the second continued fraction goes to $[0,\overline{1}]=\frac{2}{1+\sqrt{5}}$. The sum of the two terms is 
$\frac{8+2\sqrt{5}}{3+\sqrt{5}}+ \frac{2}{1+\sqrt{5}}=\frac{24+12\sqrt{5}}{8+4\sqrt{5}}=3$.
This shows $L(\za)=3$.

\medskip

\begin{problem}\label{prob 8}
The set in Problem~\ref{prob 7} is a union of sets $\cup_{(a,b)}\{L(\zo)\mid \zo\in\cald(a,b)\}$
and each of these subsets contains exactly one number smaller than 3. Are these subsets pairwise disjoint? Do they provide an interesting additional structure for the Lagrange spectrum?
\end{problem}


\bibliographystyle{amsplain}

\end{document}